\documentclass[12pt]{amsart}

\usepackage[utf8]{inputenc}

\usepackage[english]{babel}
\usepackage{hyperref}
\usepackage[utf8]{inputenc}
\usepackage{graphicx}
\usepackage{graphics}
\usepackage{amssymb}
\usepackage{amsmath}
\usepackage{amsthm}
\usepackage{tikz-cd}
\usepackage{mathrsfs}
\usepackage[colorinlistoftodos]{todonotes}
\usepackage{enumitem}
\usepackage{lmodern}
\usepackage[T1]{fontenc}
\usepackage{import}
\usepackage{wrapfig}
\usepackage{calc}
\usepackage[all]{xy}
\usepackage[margin=3.5cm]{geometry}
\usepackage{float}
\usepackage{mathtools}
\usepackage{subcaption}

\newtheorem{thm}{Theorem}[section]

\newtheorem{rmk}[thm]{Remark}

\newtheorem{lemma}[thm]{Lemma}
\def\R{\mathbb{R}}
\def\ep{\epsilon}
\def\C{\mathbb{C}}

\newtheorem*{theorem-non}{Theorem}

\def\ep{\epsilon}

\def\R{\mathbb{R}}

\def\C{\mathbb{C}}

\def\cal R{\mathcal R}

\newcommand{\jznote}[1]{#1}

\title{Strong Arnold chord conjecture via normalized capacities}

\author{Jungsoo Kang}
\email{jungsoo.kang@snu.ac.kr}
\address{Seoul National University, Department of Mathematical Sciences, Research Institute in
Mathematics, 08826 Seoul, South Korea}

\author{Jun Zhang}
\email{jzhang4518@ustc.edu.cn}
\address{The Institute of Geometry and Physics, University of Science and Technology of China, 96 Jinzhai Road, Hefei Anhui, 230026, China}

\begin{document}

\maketitle

\begin{abstract} We show that every dynamically convex toric domain in $\R^4$ admits a $C^1$-neighborhood $\mathcal U$ within the space  of star-shaped domains of $\R^4$ with the following property:~for any $X  \in \mathcal U$, every Legendrian knot in $\partial X$ admits a Reeb chord with distinct endpoints. A higher dimensional analog is also discussed.   
\end{abstract}

\vspace{-2mm}

\section{Introduction}  \label{sec-motivation} 

Let $X$ be a star-shaped domain in $\R^{2n}$, namely it is a compact subset such that the boundary $\partial X$ is smooth and transverse to all the lines through the origin. 
 The standard Liouville form $\lambda:=\frac{1}{2}\sum_{i=1}^nr_i^2d\theta_i$ on $\R^{2n}$ restricted to $\partial X$ is a contact form, where $(r_i,\theta_i)$ are the polor coordinates on each factor of $\C^n=\R^{2n}$. The Reeb vector field $R$ on $\partial X$ is defined by $\iota_R(d\lambda|_{\partial X})=0$ and $\lambda(R)=1$. A submanifold $\Lambda\subset\partial X$ is called Legendrian if $\dim\Lambda= n-1$ and $T\Lambda\subset \ker(\lambda|_{\partial X})$.
 
\jznote{ In  \cite{Moh}, Mohnke proved that, for every closed Legendrian submanifold $\Lambda \subset \partial X$, there is a Reeb chord $\gamma:[0,T]\to\partial X$ with $\gamma(0),\gamma(T)\in\Lambda$. This answers a weaker form of the Arnold conjecture in \cite[Section 8]{Arn86}. In fact, the  conjecture even asks for the existence of a genuine Reeb chord, i.e.~$\gamma(0) \neq \gamma(T)$.} We call the original form of the  conjecture in \cite{Arn86} the {\it strong} Arnold chord conjecture. There has been not much study on the strong Arnold chord conjecture. It was proved for Euclidean balls in \cite{Zil16} and for  convex domains in \cite{Kan23}. If we restrict the class of closed Legendrian submanifolds to those which admit metics with nonpositive curvature, the conjecture is true for star-shaped domains sufficiently $C^1$-close to a Euclidean ball, as established in \cite{CM18}. We also note that there is a star-shaped domain arbitrarily $C^0$-close to a Euclidean ball for which  the strong Arnold chord conjecture is false, see Hutchings' example  \cite[Remark 1.3.(b)]{Kan23}.

A goal of this note is to extend the result in \cite{CM18} mentioned above to   strictly monotone toric domains. The toric domain $X_\Omega$ associated with a domain $\Omega\subset\R^n_{\geq0}$ is defined by the preimage of $\Omega$ under the moment map
\[
\mu:\C^n\to\R^n_{\geq0},\qquad \mu(z):=\pi(|z_1|^2,\dots,|z_n|^2),
\]
that is, $X_\Omega:=\mu^{-1}(\Omega)$. A toric domain $X_\Omega$ is called monotone (or strictly monotone) if $X_\Omega$ is compact with smooth boundary and the outward normals at every point in $\partial\Omega\cap\R^n_{>0}$ have nonnegative (or positive) entries. As shown in  \cite[Proposition 1.8]{GHR}, strictly monotone toric domains are dynamically convex, i.e.~the Conley-Zehnder index of any closed Reeb orbit on $\partial X_\Omega$ is at least $n+1$. The converse is also true in dimension 4. 

To state our result, let us introduce some notations. For a star-shaped domain $X\subset\R^{2n}$, we denote by $\mathcal{L}(X)$ the set of all closed Legendrian submanifolds in $\partial X$. We also denote by  $\mathcal A_{\rm min}(\Lambda)$ the minimal period among the periods of all Reeb chords with endpoints on $\Lambda\in\mathcal{L}(X)$ and by $\mathcal A_{\rm min}( X)$ the minimal period among the periods of all periodic Reeb orbits on $\partial X$.

\begin{thm} \label{main-thm} Let $X_{\Omega} \subset \R^4$ be a dynamically convex toric domain. Then there exists a $C^1$-neighborhood $\mathcal U$ of $X_{\Omega}$ in the space of star-shaped domains in $\R^4$ such that, for every $X \in \mathcal U$ (not necessarily toric), we have 
\[ \sup_{\Lambda \in\mathcal{L}( X)} \mathcal A_{\rm min}(\Lambda) < \mathcal A_{\rm min}(X). \]
In particular, the strong Arnold chord conjecture holds for every $X \in \mathcal U$.  \end{thm}

A higher dimensional analog of the above theorem is true under a technical assumption \cite[Assumption 7.1]{Per}, which is about the existence of a virtual perturbation scheme for linearized contact homology. 

\begin{thm} \label{main-thm-2} We assume \cite[Assumption 7.1]{Per}. Let $X_{\Omega} \subset \R^{2n}$ be a strictly monotone toric domain.  Then there exists a $C^1$-neighborhood $\mathcal U$ of $X_{\Omega}$ in the space of star-shaped domains in $\R^{2n}$ such that, for every $X \in \mathcal U$ (not necessarily toric), 
\[
\sup_{\Lambda\in\mathcal{L}_{\rm torus}(X)} \mathcal A_{\rm min}(\Lambda) < \mathcal A_{\rm min}(X),
\]
where $\mathcal{L}_{\rm torus}(X)$ denotes the set of all Legendrian {\rm tori} in $\partial X$. \end{thm}



Let us end this introduction with a couple of remarks. 

\begin{rmk}\label{reg} The above results are false for $C^0$-neighborhoods due to Hutchings' example \cite[Remark 1.3.(b)]{Kan23} mentioned eariler. \end{rmk}
\begin{rmk}
The above results do not hold for monotone toric domains (that are not {\rm strictly} monotone toric). Let us consider  $X_\Omega\subset\R^4$, where $\Omega\subset\R^2_{\geq0}$ is a smooth approximation of the domain 
\[
\{(x,y)\in \R^2_{\geq0} \mid x\leq \epsilon,\; y\leq 3\}\cup \{(x,y)\in \R^2_{\geq0} \mid x+y\leq 2\}
\]
for small $\epsilon>0$. Then $A_{\rm min}(X_\Omega)=\ep$ while $\displaystyle\sup_{\Lambda \in\mathcal{L}( X_\Omega)} \mathcal A_{\rm min}(\Lambda)\geq\mathcal A_{\rm min}(\Lambda')=1$ where $\Lambda'$ is a Legendrian knot in $\mu^{-1}(1,1)\subset\partial X_\Omega$ (see the proof of Lemma \ref{lem-rel-min-cap}).
\end{rmk}

\begin{rmk}\label{concave}
	A toric domain $X_{\Omega}\subset\R^{2n}$ is called {\it concave} if $\R^{n}_{\geq 0}\setminus\Omega$ is convex (see the second picture in Figure \ref{fig_strict}). For strictly monotone and concave toric domains $X_\Omega$, the inequality in Theorem \ref{main-thm-2} can be quantified (see Remark \ref{concave2}):~for any $0<\delta<1$, there exists a $C^1$-neighborhood $\mathcal U(\delta)$ such that every $X \in \mathcal U(\delta)$ satisfies
	\[
	\sup_{\Lambda\in\mathcal{L}_{\rm torus}(X)} \mathcal A_{\rm min}(\Lambda) < \frac{1}{n\delta} \, \mathcal A_{\rm min}(X).
	\]
\end{rmk}

\begin{rmk}\label{convex}
	A toric domain $X_{\Omega}\subset\R^{2n}$ is called {\it convex} if the set $\{(x_1,\dots,x_n)\in\R^{n}\mid (|x_1|,\dots,|x_n|)\in\Omega\}$ is convex (see the third picture in Figure \ref{fig_strict}). In a sharp contrast to concave toric domains, in the set of strictly monotone and convex toric domains, $\displaystyle\sup_{\Lambda\in\mathcal{L}_{\rm torus}(X_\Omega)} \mathcal A_{\rm min}(\Lambda)$ and  $\mathcal A_{\rm min}(X_{\Omega})$ can be arbitrarily close to each other. 
\end{rmk}


\begin{rmk}Adapting the proof of \cite[Theorem 1.2]{Kan23}, it appears feasible that one can prove the following. For every strictly monotone toric domain $X_\Omega\subset\R^{2n}$ and  closed Legendrian submanifold $\Lambda \subset \partial X_{\Omega}$ (not necessarily torus), $\mathcal A_{\rm min}(\Lambda) < \mathcal A_{\rm min}(X_{\Omega})$. This would answer the strong Arnold chord conjecture for these $X_{\Omega}$. However, this inequality is not sufficient to draw the same conclusion for star-shaped domains $C^1$-close to $X_\Omega$. 
\end{rmk}

%
%

\subsection*{Acknowledgements} This project was initiated by second author's academic visit to Seoul National University (SNU) in October 2023. He thanks the warm invitation from the first author, as well as the active research atmosphere at SNU. The first author is supported by National Research Foundation of Korea grant 2020R1A5A1016126 and RS-2023-00211186. He thanks the second author for his warm hospitality at USTC-IGP. The second author is partially supported by National Key R\&D Program of China No.~2023YFA1010500, NSFC No.~12301081, NSFC No.~12361141812, and USTC Research Funds of the Double First-Class Initiative. The second author also thanks useful communications with Zhengyi Zhou. 

\section{Proofs} \label{sec-proof}

We recall from \cite[Section 2]{GH} the following observation for a compact toric domain $X_\Omega\subset\R^{2n}$ with smooth boundary. Let $w=(w_1,\dots,w_n)\in\partial\Omega$, and let $\nu=(\nu_1,\dots,\nu_n)$ be an outward normal vector of $\Omega$ at $w$. The Reeb vector field $R$ on $\partial X_\Omega$ at $z\in\mu^{-1}(w)$ is  given by
\begin{equation}\label{reeb}
R_z=\frac{2\pi}{w\cdot \nu}\sum_{w_i\neq0}\nu_i\frac{\partial}{\partial\theta_i}=\frac{2\pi}{w\cdot \tilde\nu}\sum_{i=1}^n\tilde\nu_i\frac{\partial}{\partial\theta_i},
\end{equation}
where $\tilde \nu:=(\tilde\nu_1,\dots,\tilde\nu_n)$ is defined by $\tilde\nu_i=\nu_i$ if $w_i\neq 0$ and $\tilde \nu_i=0$ otherwise. If $\tilde\nu$ is a scalar multiple of an integral vector, then the torus $\mu^{-1}(w)$ is foliated by closed Reeb orbits with period $w\cdot \nu=w\cdot\widetilde \nu$. 

In general, the minimal period $\mathcal A_{\rm min}$ does not behave like a symplectic capacity;  a symplectic embedding $X \stackrel{s}{\hookrightarrow} Y$ of two star-shaped domains $X$ and $Y$ may badly violate the monotonicity property $\mathcal A_{\rm min}(X) \leq \mathcal A_{\rm min}(Y)$. However, it is known that, for convex domains, $\mathcal A_{\rm min}$ is a symplectic capacity. The following lemma shows that the same holds for strictly monotone toric domains. We recall that the Gromov width (also known as the ball capacity) is defined by 
\[
c_{\rm Gr}(X):=\sup \{r>0\mid B^{2n}(r)\stackrel{s}{\hookrightarrow} X\},\quad\text{where}\;\; B^{2n}(r):=\{z\in \C^n \;|\; \pi|z|^2\leq r\}.
\] 

\begin{lemma} \label{lem-abs-min-cap} For every  strictly monotone toric domain  $X_{\Omega}$, $\mathcal A_{\rm min}(X_{\Omega})$ coincides with $c_{\rm Gr}(X_{\Omega})$. Moreover, it coincides with any other normalized symplectic capacity.  
\end{lemma}

\begin{proof} The second conclusion is due to \cite[Theorem 1.1]{CG-H-c}, which proves that, on a monotone toric domain (not necessarily strictly monotone), all normalized symplectic capacities coincide. The proof of \cite[Theorem 1.1]{CG-H-c} shows that $c_{\rm Gr}(X_{\Omega}) = R$ is achieved by the inclusion  $B^{2n}(R)\subset X_{\Omega}$. 
Note that $\partial B^{2n}(R)$ intersects $\partial X_\Omega$ along a torus family of closed Reeb orbits of period $R$.  

Now, we view $B^{2n}(R) = X_{\Delta}$ as the toric domain associated with the standard simplex $\Delta \subset \R^n$ rescaled by $R$. Since $B^{2n}(R)\subset X_{\Omega}$,  $\Delta\subset\Omega$ and thus every $w\in \partial\Omega$ satisfies $w_1 + \cdots +w_n  \geq R$, where the equality holds exactly when $w\in\partial\Omega\cap\partial\Delta$. 
Let $w \in \partial \Omega $ be a point such that the vector $\tilde\nu = (\tilde\nu_1, ..., \tilde\nu_n)$ modified from an outward normal vector $\nu$ at $w$, see above, is a scalar multiple of an integral vector. Without loss of generality, we may assume that $\tilde\nu$ is integral and the greatest common divisor of $\tilde\nu_1,\dots,\tilde\nu_n$ equals $1$. Then the torus  $\mu^{-1}(w)$ is foliated by closed Reeb orbits of period $w\cdot\tilde\nu$. 
 Moreover, by the definition of  $X_\Omega$ being strictly monotone, $\widetilde\nu_i \geq 1$ unless $w_i=0$, and the equalities hold for $w\in\partial\Omega\cap\partial\Delta$. Hence, $w \cdot \tilde\nu\geq w_1 + \cdots +w_n  \geq R$, and the equality holds for $w\in\partial\Omega\cap\partial\Delta$. 
This proves  $\mathcal A_{\rm min}(X_{\Omega}) = R = c_{\rm Gr}(X_{\Omega})$. \end{proof}

Next, we show that, similar to Lemma \ref{lem-abs-min-cap}, $\sup_{\Lambda}\mathcal A_{\rm min}(\Lambda)$ over Legendrian tori $\Lambda$ coincides with a symplectic capacity satisfying a different normalization condition for strictly monotone toric domains. To state this precisely, recall from \cite{CM18} that the Lagrangian capacity $c_{\rm Lag}$ of a start-shaped domain $X\subset\R^{2n}$ is defined by
\begin{equation} \label{dfn-Lag-cap} 
c_{\rm Lag}(X) : = \sup_{L} \inf \Big\{\int_\sigma\omega \;\Big|\;\sigma\in\pi_2(X,L),\;\int_\sigma\omega >0 \Big\},
\end{equation}
where the supremum ranges over all Lagrangian tori $L$ contained in $X$. 
Under the assumption on the existence of a virtual perturbation scheme in \cite[Assumption 7.1]{Per}, the Lagrangian capacity is computed for convex or concave toric domains in \cite[Theorem 7.65]{Per}. In particular, $c_{\rm Lag}(P^{2n}(a)) = c_{\rm Lag}(N^{2n}(a)) = a$, where 
\[
\begin{split}
	P^{2n}(a)&:=\{z\in \C^n \mid \pi|z_i|^2\leq a\; \text{ for all } 1\leq i\leq  n \},\\
	N^{2n}(a)&:=\{z\in \C^n \mid \pi|z_i|^2\leq a\; \text{ for some } 1\leq i\leq  n \}.
\end{split}
\]
Therefore, according to Definition 4 in \cite{GPR}, $c_{\rm Lag}$ is a cube normalized capacity. Another cube normalized capacity is defined by $c_{\square}(X):=\sup \{a>0\mid P^{2n}(a)\stackrel{s}{\hookrightarrow} X \}$.

\begin{lemma} \label{lem-rel-min-cap} We assume \cite[Assumption 7.1]{Per}. Then,  for every strictly monotone toric domain $X_{\Omega}\subset\R^{2n}$, we have 
\[
\sup_{\Lambda\in \mathcal{L}_{\rm torus}(X_{\Omega})} \mathcal A_{\rm min}(\Lambda)=c_{\rm Lag}(X_{\Omega})=c_{\square}(X_{\Omega}).
\]
Accordingly, $\displaystyle\sup_{\Lambda\in \mathcal{L}_{\rm torus}(X_{\Omega})}\mathcal A_{\rm min}(\Lambda)$ coincides with any cube normalized capacity.  \end{lemma}

\begin{proof} The second conclusion follows from \cite[Theorem 2]{GPR}, which proves that, on a monotone toric domain (not necessarily strictly monotone), all cube normalized symplectic capacities coincide. In particular, $c_{\rm Lag}(X_{\Omega})=c_{\square}(X_{\Omega})$. Moreover, $c_{\square}(X_{\Omega}) = a$ is achieved by the inclusion $P^{2n}(a)\subset X_{\Omega}$, see \cite[Theorem 11]{GPR}. 

The proof of \cite[Theorem 4]{Moh} shows
\begin{equation} \label{Moh-inequality}
\sup_{\Lambda\in \mathcal{L}_{\rm torus}(X_{\Omega})} \mathcal A_{\rm min}(\Lambda) \leq c_{\rm Lag}(X_{\Omega}).
\end{equation}
In what follows, we prove the opposite inequality.  We consider the Lagrangian torus $L_a:=\mu^{-1}(a, ..., a)$, which is  contained in $\partial X_\Omega$ since $P^{2n}(a)\subset X_\Omega$ and $c_{\square}(X_{\Omega}) = a$ as mentioned above. This Lagrangian torus is foliated by isotopic Legendrian tori $\Lambda_a\subset L_a$, and we will prove $\mathcal{A}_{\min}(\Lambda_a)=a$. The case of $\dim X_{\Omega} = 4$ and $\dim L_a=2$ is illustrated in Figure \ref{fig_Leg_torus}, where $L_a$ is pictured by a square in the coordinates $(\theta_1, \theta_2)$ with facing edges identified.
 \begin{figure}[h]
\includegraphics[scale=0.8]{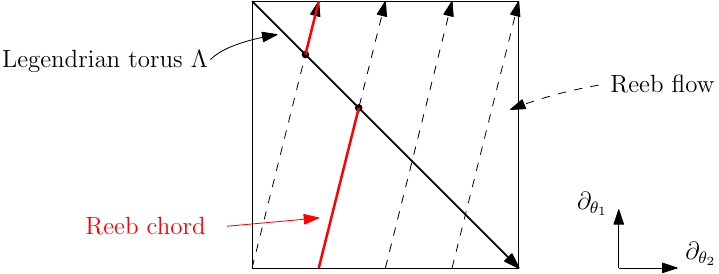}
\centering
\caption{Legendrian torus $\Lambda$ and Reeb chord.} \label{fig_Leg_torus}
\end{figure}
Let $\nu$ be an outward normal vector of $\Omega$ at $a$. The contact structure $\xi:=\ker(\lambda|_{\partial X_{\Omega}})$ at $z\in L_a$ can be described as
 	\[
 	\xi_z= \Big\{\sum_{i=1}^n \Big( x_i\frac{\partial}{\partial\mu_i}+y_i\frac{\partial}{\partial\theta_i}\Big) \;\Big|\; \sum_{i=1}^n\nu_ix_i=0,\; \sum_{i=1}^n y_i=0\Big\},
 	\]
 	where $\mu_i:=\pi r_i^2$, see \cite[Section 2]{GH}.
 	Thus, for any $0\leq c< 2\pi$, the torus $\Lambda_a\subset L_a$ defined by $\sum_{i=1}^n\theta_i=c$ is Legendrian. The negative diagonal in Figure \ref{fig_Leg_torus} corresponds to the Legendrian knot $\theta_1+\theta_2=0$. Recall from \eqref{reeb} that the Reeb vector field at $z\in L_a$ has the form
 	$
	R_z=\frac{2\pi}{a\cdot \nu}\sum_{i=1}^n\nu_i\frac{\partial}{\partial\theta_i}.
	$
 	A straightforward computation shows that every minimal Reeb chord with endpoints on $\Lambda_a$ has period $a$. 
Therefore, 
\[ \sup_{\Lambda\in\mathcal{L}_{\rm torus}(X_{\Omega})} \mathcal A_{\rm min}(\Lambda)  \leq c_{\rm Lag}(X_{\Omega}) = c_{\square}(X_{\Omega})=\mathcal{A}_{\rm min}(\Lambda_a) \leq \sup_{\Lambda\in\mathcal{L}_{\rm torus}(X_{\Omega})}  \mathcal A_{\rm min}(\Lambda), \]
and this completes the proof. \end{proof}

\begin{rmk} \label{rmk-no-assumption}The proof of Lemma \ref{lem-rel-min-cap} relies on the fact that the Lagrangian capacity $c_{\rm Lag}$ is cube normalized, which holds as of now under the technical assumption \cite[Assumption 7.1]{Per}. However, thanks to automatic transversality, we do not need this assumption in dimension 4 (see \cite[Theorem 6.41]{Per} and \cite[Remark 8]{GPR}). 
\end{rmk}

Finally, we are ready to give the proof of Theorem \ref{main-thm-2}. 

\begin{proof}[Proof of Theorem \ref{main-thm-2}] As mentioned above,   $c_{\square}(X_{\Omega})=a$ is achieved by the inclusion $P^{2n}(a)\subset X_{\Omega}$. Since $B^{2n}(a) \subset P^{2n}(a)$ and $X_{\Omega}$ is strictly monotone, there exists some $\kappa_{\Omega}>0$ (only depending on the moment image $\Omega$) such that $B^{2n}(a+\kappa_{\Omega}) \subset X_{\Omega}$. For a schematic picture, see the leftmost picture in Figure \ref{fig_strict}. 
Therefore, we have  
\begin{equation} \label{strict-inequality} 
c_{\square}(X_{\Omega}) = a < a+ \kappa_{\Omega} \leq c_{\rm Gr}(X_{\Omega}),
\end{equation}
and this in turn yields
\begin{equation} \label{gap}
\sup_{\Lambda\in\mathcal{L}_{\rm torus}(X_\Omega)} \mathcal A_{\rm min}(\Lambda) = c_{\rm Lag}(X_{\Omega}) = c_{\square}(X_{\Omega}) \leq  c_{\rm Gr}(X_{\Omega}) - \kappa_{\Omega} = \mathcal A_{\rm min}(X_{\Omega}) - \kappa_{\Omega}, 
\end{equation}
where the first two equalities come from Lemma \ref{lem-rel-min-cap} and the last equality is due to Lemma \ref{lem-abs-min-cap}. In particular, we obtain 
\begin{equation}\label{strict_ineq}
	\sup_{\Lambda\in\mathcal{L}_{\rm torus}( X_\Omega)} \mathcal A_{\rm min}(\Lambda) < \mathcal A_{\rm min}(X_{\Omega}).
\end{equation}
\begin{figure}[h]
\includegraphics[scale=0.72]{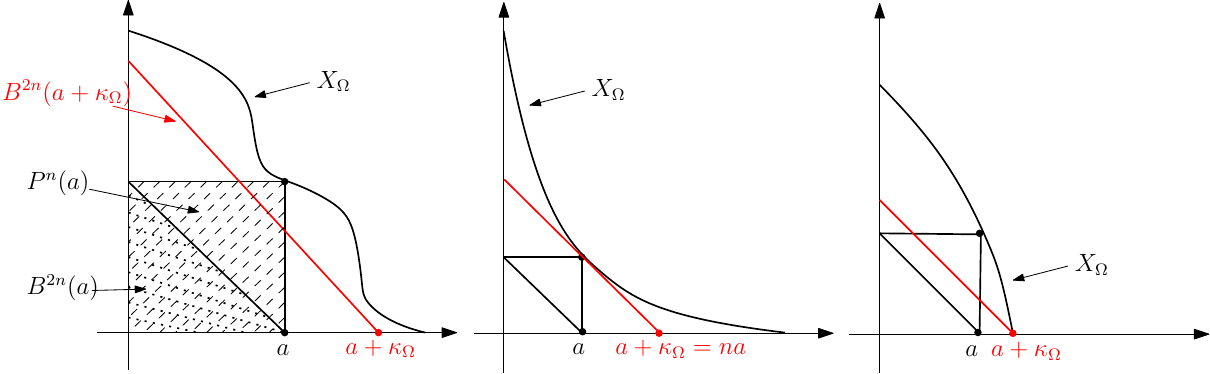}
\centering
\caption{$B^{2n}(a)$ and $P^{2n}(a)$ in $X_{\Omega}$.} \label{fig_strict}
\end{figure}

Next, we observe that \eqref{strict_ineq} remains true for star-shaped domains sufficiently $C^1$-close to $X_\Omega$. 
This follows from the simple observation that, for any $\ep>0$, there exists a $C^1$-neighborhood $\mathcal U(\ep)$ of $X_\Omega$ in the space of star-shaped domains with the following two properties. First, for every $X\in\mathcal U(\ep)$, 
\begin{equation} \label{per-1}
\mathcal A_{\rm min}(X) > \mathcal A_{\rm min}(X_{\Omega}) - \ep.
\end{equation} 
This readily follows from the Arzela-Ascoli theorem. Second, for every $X \in \mathcal U(\ep)$, \begin{equation} \label{per-2}
\sup_{\Lambda\in\mathcal{L}_{\rm torus}( X)} \mathcal A_{\rm min}(\Lambda) \leq c_{\rm Lag}(X) < c_{\rm Lag}(X_\Omega)+\ep. 
\end{equation}
The first inequality follows from \cite[Theorem 4]{Moh}, as pointed out in \eqref{Moh-inequality}, and the second one is due to the fact that $c_{\rm Lag}$ is continuous with respect to the Hausdorff metric, and thus also to the $C^1$-topology. 

We take $\ep = \frac{\kappa_{\Omega}}{2}$. Then (\ref{gap}), (\ref{per-1}), and (\ref{per-2}) together imply that, for every star-shaped domain $X \in \mathcal U(\frac{\kappa_{\Omega}}{2})$, we have 
\[ 
\sup_{\Lambda\in\mathcal{L}_{\rm torus}( X)} \mathcal A_{\rm min}(\Lambda)  < c_{\rm Lag}(X_{\Omega}) + \frac{\kappa_{\Omega}}{2} \leq \mathcal A_{\rm min}(X_{\Omega}) - \frac{\kappa_{\Omega}}{2} < \mathcal A_{\rm min}(X). \]
This completes the proof, with the desired $C^1$-neighborhood $\mathcal U=\mathcal U(\frac{\kappa_{\Omega}}{2})$. \end{proof}

\begin{proof}[Proof of Theorem \ref{main-thm}] In dimension 4, the assumption \cite[Assumption 7.1]{Per}  can be dropped as explained in Remark \ref{rmk-no-assumption}, and a toric domain is dynamically convex if and only if it is strictly convex. Thus Theorem \ref{main-thm} follows from Theorem \ref{main-thm-2}. 
\end{proof}

\begin{rmk}\label{concave2}
	For a concave toric domain $X_{\Omega}\subset\R^{2n}$,  the estimate (\ref{strict-inequality}) can be improved by $nc_{\square}(X_{\Omega})=c_{\rm Gr}(X_{\Omega})$ since if $P^{2n}(a) \subset X_{\Omega}$ realizes $c_{\square}(X_{\Omega})$, then $B^{2n}(na) \subset X_{\Omega}$ realizes $c_{\rm Gr}(X_{\Omega})$. The claimed inequality in Remark \ref{concave} follows.
\end{rmk}

\bibliographystyle{amsplain}
\bibliography{biblio_AC_DC}

\providecommand{\bysame}{\leavevmode\hbox to3em{\hrulefill}\thinspace}
\providecommand{\MR}{\relax\ifhmode\unskip\space\fi MR }
\providecommand{\MRhref}[2]{%
  \href{http://www.ams.org/mathscinet-getitem?mr=#1}{#2}
}
\providecommand{\href}[2]{#2}
\begin{thebibliography}{10}

\bibitem{Arn86}
Vladimir~I. Arnol'd, \emph{First steps in symplectic topology}, Russian
  Mathematical Surveys \textbf{41} (1986), no.~6, 1.

\bibitem{CM18}
Kai Cieliebak and Klaus Mohnke, \emph{Punctured holomorphic curves and
  {L}agrangian embeddings}, Invent. Math. \textbf{212} (2018), no.~1, 213--295.
  \MR{3773793}

\bibitem{GH}
Jean Gutt and Michael Hutchings, \emph{Symplectic capacities from positive
  {$S^1$}-equivariant symplectic homology}, Algebr. Geom. Topol. \textbf{18}
  (2018), no.~6, 3537--3600. \MR{3868228}

\bibitem{GHR}
Jean Gutt, Michael Hutchings, and Vinicius G.~B. Ramos, \emph{Examples around
  the strong {V}iterbo conjecture}, J. Fixed Point Theory Appl. \textbf{24}
  (2022), no.~2, Paper No. 41, 22. \MR{4413022}

\bibitem{GPR}
Jean Gutt, Miguel Pereira, and Vinicius Ramos, \emph{Cube normalized symplectic
  capacities}, arXiv preprint, arXiv: 2208.13666.

\bibitem{CG-H-c}
Richard Hind and Dan Cristofaro-Gardiner, \emph{On the agreement of symplectic
  capacities in high dimension}, arXiv preprint, arXiv: 2307.12125.

\bibitem{Kan23}
Jungsoo Kang, \emph{On the strong {A}rnold chord conjecture for convex contact
  forms}, arXiv preprint, arXiv: 2304.07016.

\bibitem{Moh}
Klaus Mohnke, \emph{Holomorphic disks and the chord conjecture}, Ann. of Math.
  (2) \textbf{154} (2001), no.~1, 219--222. \MR{1847594}

\bibitem{Per}
Miguel~B. Pereira, \emph{Equivariant symplectic homology, linearized contact
  homology and the lagrangian capacity}, arXiv preprint, arXiv: 2205.13381.

\bibitem{Zil16}
Fabian Ziltener, \emph{On the strict {A}rnold chord property and coisotropic
  submanifolds of complex projective space}, Int. Math. Res. Not. IMRN (2016),
  no.~3, 795--826. \MR{3493434}

\end{thebibliography}

\medskip

\end{document}